\documentclass[a4paper,american,10pt]{amsart}

\usepackage{amsmath,amsthm}
\usepackage{amssymb}
\usepackage{caption}
\usepackage{subcaption}
\usepackage{amsfonts}
\usepackage{amscd}
\usepackage{wrapfig}
\input xy
\xyoption{all}

\usepackage{listings} 
\usepackage{xcolor}
\usepackage{colortbl}
\definecolor{cobalt}{rgb}{0.0, 0.28, 0.67}

\usepackage{hyperref}
\usepackage{tikz-cd}
\usepackage{tikz}
\usetikzlibrary{trees,positioning,calc,arrows.meta}
\usepackage{tikz-qtree}

\theoremstyle{definition}
\newtheorem{definition}{Definition}
\newtheorem{lemma}{Lemma}
\newtheorem{theorem}{Theorem}
\newtheorem{example}{Example}
\newtheorem{remark}{Remark}

\newtheorem{corollary}{Corollary}

\newcommand\ZZ{\mathbb Z}

\newcommand\RR{\mathbb R}
\newcommand\QQ{\mathbb Q}
\newcommand\vect[2]{\tiny
\begin{pmatrix}
#1\\
#2
\end{pmatrix}
}
\newcommand\arctanh{\operatorname{arctanh}}
\newcommand\arcsinh{\operatorname{arcsinh}}

\newcommand\topo[1]{\begin{tikzpicture}[scale=0.18,baseline=(current bounding box.center)]
  \draw[thick] (0,0) -- (-1.2,1.8);
  \draw[thick] (0,0) -- (-1.2,-1.8);
  \draw[thick] (0,0) -- (2.4,0);
  \node at (-1.5,0) {\small $r$};
  \node at (1.4,1.1) {\small $s$};
  \node at (1.4,-1.1) {\small $t$};
  \node at (4.3,0) {\small$\in \mathcal{#1}$};
\end{tikzpicture}}

\tikzset{
  small latex/.tip={Latex[scale=0.6]}  
}

\newcommand\topoo[1]{%
  \begin{tikzpicture}[baseline={(0,-0.5ex)}, scale=0.1, thick]
    \draw[-{small latex}] (0,0) -- (-1.5,2.5);
    \draw[-{small latex}] (0,0) -- (-1.5,-2.5);
    \draw[-{small latex}] (0,0) -- (3,0);

    \node[red] at (-2,2.3) {\tiny $_f$};
    \node[red] at (-2,-2.3) {\tiny $_e$};
    \node[red] at (2,1.3) {\tiny $_g$};

    \node at (7,0) {$\in \mathcal{#1}$};
  \end{tikzpicture}%
}

\usepackage[foot]{amsaddr}

\begin{document}



\author{Nikita Kalinin}
\title[Telescoping over topographs]{Evaluation of lattice sums via telescoping over topographs}
 \address{Guangdong Technion Israel Institute of Technology (GTIIT),
241 Daxue Road, Shantou, Guangdong Province 515603, P.R. China,  Technion-Israel Institute of Technology, Haifa, 32000, Haifa district, Israel, nikaanspb@gmail.com}
\maketitle

\begin{abstract}
Topographs, introduced by Conway in 1997, are infinite trivalent planar trees used to visualize the values of binary quadratic forms. In this work, we study series whose terms are indexed by the vertices of a topograph and show that they can be evaluated using telescoping sums over its edges. 

Our technique provides arithmetic proofs for modular graph function identities arising in string theory, yields alternative derivations of Hurwitz-style class number formulas, and provides a unified framework for well-known Mordell-Tornheim series and Hata's series for the Euler constant $\gamma$.

Our theorems are of the following spirit: we cut a topograph along an edge (called the {\it root}) into two parts, and then sum $\frac{1}{rst}$ (the reciprocal of the product of labels on regions adjacent to a vertex) over all vertices of one part. We prove that such a sum is equal to an explicit expression depending only on the root and the discriminant of the topograph. 


\end{abstract}

\tableofcontents

\section{Introduction}

The purpose of this article is to introduce and develop the {\it telescoping-over-topograph} method and to connect it to the works on class number formulas for quadratic number fields (a classical one \cite{hurwitz1905darstellung} of Adolf Hurwitz and two recent ones:  \cite{duke2021class} of Duke, Imamo\=glu, T\'oth and \cite{o2024topographs} by O'Sullivan) as well as to the modular graph functions in the low energy genus-one expansion of type II string amplitudes due to d'Hoker, Green, G\"urdo\u{g}an, and Vanhove \cite{d2015modular,d2017modular}. 

The story begins with Hurwitz's 1905 formula  \eqref{hurwitz}, which expresses the number of $SL(2,\ZZ)$-equivalence classes of integer binary quadratic  forms with negative discriminant as a certain infinite sum. After nearly a century of neglect, Duke--Imamo\=glu--T\'oth established the corresponding positive-discriminant formula \eqref{duke}. O'Sullivan then recast these identities as vertex sums on Conway's topographs. In a parallel development, Zagier computed the lattice sum \eqref{eq_2} -- essentially the  keystone step in Hurwitz's argument -- which later proved to be the prototypical example among modular graph functions in string theory, as studied by d'Hoker, Green, G\"urdo\u{g}an, Vanhove, and others. 

Our contribution is a unifying and elementary viewpoint: many of these results fall out from a simple geometric telescoping argument. Below, we explain the context in greater detail, in historical order, starting from Hurwitz and moving forward.

\subsection{Formulas related to the class number} For $v=(m,n)$, let $q=[A,B,C]$ denote a quadratic form $$q(v)=Am^2+Bmn + C n^2.$$ Two binary quadratic forms $[A,B,C]$ and $[A',B',C']$ with integer coefficients are called equivalent if there exist $a,b,c,d\in\mathbb Z$ with $ad-bc=1$ such that 

$$A'm^2+B'mn+C'n^2 = A(am+bn)^2+B(am+bn)(cm+ dn)+C(cm+ dn)^2.$$
 That is,
\begin{equation*}
[A,B,C]\sim [Aa^2 +Bac+Cc^2, 2Aab+B(ad+bc)+2Ccd, Ab^2+Bbd+Cd^2].
\end{equation*}
In particular,
\begin{equation}
\label{equiv2}
[1,0,1]\sim [a^2 +c^2, 2ab+2cd, b^2+d^2] = [A,B,C], \text{ and } 
\end{equation}
$$A+B+C = (a+b)^2+(c+d)^2 \text{ in this case}.$$
The {\it discriminant} $D$ of a binary quadratic form $[A, B, C]$ is $B^2-4AC$ and it is invariant under the above equivalence. An integer $D$ is called a {\it fundamental} discriminant if either (a) $D$ is square-free and $D\equiv 1\pmod 4$, or (b) $D\equiv 0\pmod 4$, $D/4$ is square-free and $D/4\equiv 2,3\pmod 4$.

Denote by $h(D)$ the number of equivalence classes of forms with $\gcd(A, B, C)=1$ and discriminant $D$. Although formulas for $h(D)$ exist, it remains difficult to estimate its asymptotic behavior; for example, it is still an open problem to prove that there exist infinitely many values of $D>0$ for which $h(D)=1$.

In 1905, Adolf Hurwitz wrote a paper on an infinite series representation of the class number $h(D)$ in the positive-definite case.

\begin{theorem}[Hurwitz,\cite{hurwitz1905darstellung}] For $D<0$, a fundamental discriminant, all the terms  in the following sum are positive and this sum converges to $h(D)$:
\begin{equation}
\label{hurwitz}
h(D) = \frac{\omega_D}{12\pi}|D|^{3/2}\sum\limits_{\substack{A>0 \\ B^2-4AC=D}}\frac{1}{A(A+B+C)C}, \end{equation}
where 
\begin{equation}
 \omega_D=
\begin{cases}
1 \text{ for } D<-4,\\
 2 \text{ for } D=-4,\\
 3 \text{ for } D=-3.
\end{cases}
\end{equation}
\end{theorem}

Hurwitz's proof amounts to computing the area of a certain domain in two different ways, see Section~\ref{sec_duality} for details.
 
\begin{example}
For the quadratic form $q(v)=\|v\|^2=m^2+n^2$, the discriminant $D=-4$, and $ h(D)=1$. Hurwitz's arguments \cite[p. 20]{hurwitz1905darstellung} lead to the formula 
\begin{equation}
\label{sum_1}
\sum_{\substack{a,b,c,d\in \ZZ_{\geq 0},\\  ad-bc=1}} \frac{1}{(a^2+b^2)(c^2+d^2)((a+c)^2+(b+d)^2)} = \frac{\pi}{4},
\end{equation}

By the change of variables \eqref{equiv2}, it follows that
$$\frac{1}{A(A+B+C)C} = \frac{1}{(a^2+b^2)((a+c)^2+(b+d)^2)(c^2+d^2)},$$
 thus, putting it all together and multiplying by the constant $3$ coming from the three cyclic orders of the denominators of the terms, we rewrite \eqref{sum_1} as

$$h(-4)=1= \frac{2}{12\pi}\cdot 2^3\cdot \frac{3\pi}{4}= \frac{\omega_{-4}}{12\pi}\cdot|-4|^{3/2}\sum\limits_{\substack{A>0 \\ B^2-4AC=-4}}\frac{1}{A(A+B+C)C}.$$
\end{example}
Hurwitz's result has been largely unnoticed for nearly a century, with the only exceptions: Dickson's {\it History of number theory} \cite[p.167] {dickson1952history} in 1952 and Sczech's work \cite{sczech1992eisenstein} on Eisenstein cocycles for $GL_2\QQ$ in 1992.

It was revived in 2019, in the Duke-Imamo\=glu-T\'oth paper \cite{duke2021class}, and a formula for the indefinite case ($D>0$), similar to \eqref{hurwitz},  was established:

\begin{theorem}[\cite{duke2021class}, Theorem 3, p. 3997] For $D>0$, a fundamental discriminant,
\begin{equation}\label{duke}
h(D)\log \varepsilon_D = D^{1/2}\sum\limits_{\substack{[A,B,C]\text{  reduced}\\ B^2-4AC=D}}\frac{1}{B}+D^{3/2}\sum_{\substack{A,C,A+B+C>0\\ B^2-4AC=D}}\frac{1}{3(B+2A)B(B+2C)}.
\end{equation}
\end{theorem}
Here $\varepsilon_D$, the fundamental unit, is defined as $\varepsilon_D:= (t_D+u_D\sqrt D)/2$, where $(t_D, u_D)$ is the smallest solution to $t^2-Du^2=4$ in positive integers.

In 1997, in his book {\it The Sensual (Quadratic) Form} \cite{conway1997sensual}, John H. Conway introduced topographs, a graphical tool for visualizing binary quadratic forms and their values over the integers.  A topograph provides an intuitive and surprisingly powerful visual framework for understanding, for example, reduction algorithms. 

A topograph for a binary quadratic form $q$ is an infinite trivalent planar tree $\mathcal T$ with labels in the connected components (regions) of $\RR\setminus \mathcal{T}$. Each region corresponds bijectively a pair $(v,-v)$ of primitive lattice vectors in $\ZZ ^2$ (i.e., a point in $ P\mathbb {Q} ^2$), and the label on this region is the value $q(v)=q(-v)$ of the quadratic form $q$. At each vertex, the three adjacent regions correspond to three primitive vectors $v,w,v+w$, forming a basis of $\ZZ^2$. Thus, near each vertex of $\mathcal T$ the labels $r,s,t$ on regions are exactly $q(v),q(w),q(v+w)$ for a certain  basis $(v,w)$ of $\ZZ^2$. Topographs are related to many objects in mathematics \cite{buchstaber2019conway} and can also be used to study variations of Markov triples; see the popular exposition \cite{veselov2023conway}.

Note that for the basis $(v,w)$, $v=\vect{1}{0}, w=\vect{0}{1}$ we obtain $$A(A+B+C)C=q(\vect{1}{0})q(\vect{0}{1})q(\vect{1}{1}).$$ 
Thus, \eqref{hurwitz} may be interpreted as the sum of $$\frac{1}{q(\vect{1}{0})q(\vect{0}{1})q(\vect{1}{1})}$$ over all forms $q$ of discriminant $D$. Instead of applying the $SL(2,\ZZ)$ action to a form, we may apply it to a basis, thereby obtaining a sum of the form $$\frac{1}{q(v)q(w)q(v+w)}$$ over a topograph. 

We will consider summations indexed by vertices $V$ of a half of a topograph, and we sum the reciprocal of the product $|rst|$ of labels on the adjacent regions to $V$ or the product $|egf|$ of labels on the adjacent edges to $V$. 

In 2024, O'Sullivan proposed in \cite{o2024topographs} a unifying approach to these class number series via topographs. In particular,  \eqref{sum_1} was rewritten in the language of topographs as

\begin{theorem}[\cite{o2024topographs}, Theorem 9.1]\label{th_three}
Let $\mathcal{T}$ be any topograph of discriminant $D < 0$. Then
\begin{equation} \label{top}
 |D|^{3/2}\sum_{\substack{
\topo{T}
}} 
\frac{1}{|rst|} = 4\pi, 
\end{equation}
where we sum over all vertices of $\mathcal{T}$, each vertex contributing one term; here $r,s,t$ denote the labels on regions adjacent to a given vertex of $\mathcal{T}$, as explained in Section~\ref{sec_topographs}.
\end{theorem}
Then, \eqref{hurwitz} follows because the RHS of \eqref{hurwitz} is essentially the sum over all vertices of all topographs of discriminant $D$, and the number of topographs of discriminant $D$ is $h(D)$.
In the same paper, the formula \eqref{duke} for the class number in the indefinite case $D>0$ was rewritten in terms of topographs as 

\begin{theorem}[\cite{o2024topographs}, Theorem 9.2] \label{mt}
Let $\mathcal T$ be any topograph of a non-square discriminant $D>0$. Define $\mathcal T_\star$ to be equal to $\mathcal T$ except that all the river edges are relabeled with $\sqrt{D}$. Then
\[
D^{3/2}
\sum_{
\topoo{T_\star}
}
\frac{1}{|e f g|} = 2\log \varepsilon_D,
\]
where we sum over all vertices of $\mathcal T_\star$ modulo the river period (each vertex contributing one term, see \cite{o2024topographs} for details on the river and its periodicity), and $e,f,g$ are labels on the edges. 
\end{theorem}

\subsection{String-theory formulas} There was an independent parallel story. Beyond number theory, similar lattice sums arise in the study of modular graph functions in string theory. In 2008, in an unpublished note \cite{zagiernotes}, Zagier considers

\[
D_{1,1,1}(z)=\sideset{}{'}\sum_{\omega_1+\omega_2+\omega_3=0}\frac{\operatorname{Im}(z)^3}{|\omega_1\omega_2\omega_3|^{2}},\quad \omega_1,\omega_2,\omega_3\in \ZZ z+\ZZ,
\]
where $\sideset{}{'}\sum$ denotes the summation over the terms with nonzero denominators, and proves

\begin{equation}
\label{eq_2}
D_{1,1,1}(z) = 2E_3(z)+\pi^3\zeta(3),
\end{equation}

Zagier's proof involves analytic manipulations, partial telescoping, and reduction of the question to sums of $\frac{1}{(z+n)(z+m)}$, then to sums involving the real part of $\frac{1+q}{1-q}$ for $q=e^{2\pi iz}$. Finally, Zagier proves \eqref{eq_2} up to a holomorphic $SL(2,\ZZ)$-invariant function that is small at infinity (and therefore identically zero).   


Let us specialize $z=i$, so $\ZZ z+\ZZ$ becomes $\ZZ^2$. Then, one can rewrite $D_{1,1,1}(i)$ via a sum over $\omega_1,\omega_2$ spanning parallelograms of area one and hence, up to standard transformations, \eqref{eq_2} is equivalent to \eqref{sum_1}; see \cite{kalinin2024mordell} for details.


In 2017, in \cite{d2017modular}, modular graph functions (MGF) were defined; they appear in the low-energy expansion of genus-one Type II superstring amplitudes.  In perturbative type-II superstring theory, the genus-one four-graviton amplitude can be written as an integral over the torus moduli space of
products of Green functions. The integrals turn into lattice sums that are modular invariant by construction.  At weight two, one obtains classical Eisenstein series $E_n(z)$. At weight three, the unique connected vacuum diagram is a ``sunset'' graph with two vertices joined by three
propagators, and the corresponding lattice sum is precisely
\(D_{1,1,1}(z)\), as defined above.  In other words, Zagier's arithmetic
identity \eqref{eq_2}
provides the first closed formula for a nontrivial modular graph function that appears in the low-energy expansion of superstring
amplitudes.  

A general modular graph function is constructed from a graph $\Gamma$. Assign to each edge $e\in \Gamma$ the variable $\frac{y}{|\omega_e|^2}$. Consider the sum of products $\prod\limits_{e\in G} \frac{y}{|\omega_e|^2}$ over all tuples of $\omega_e\in \ZZ z+\ZZ$ such that, at each vertex, the sum of incoming $\omega_e$  equals zero.

If \(\Gamma\) consists of two vertices and $k$ edges between them, then

\begin{equation}\label{eq_string}
D_{\Gamma}(z)=
\sideset{}{'}\sum_{\substack{\omega_{1},\dots ,\omega_{k}\in\Lambda\\ \omega_{1}+\dots+\omega_{k}=0}}
\frac{y^{\,k}}{\prod_{i=1}^{k}|\omega_{i}|^{2}},\qquad
\Lambda=\ZZ+\ZZ z,\; z=x+iy,
\end{equation}
So we obtain the definition of $D_{1,1,1}$ when $\Gamma$ is a graph with two vertices and three edges ($k=3$, the ``sunset'' graph). 
The study of modular graph functions has revealed a rich network of differential and algebraic relations between them, zeta-values, Eisenstein functions, etc., resulting in hundreds of articles. 
Our telescopic viewpoint furnishes a parallel, purely arithmetical derivation of some of those relations.
For example, using \eqref{eq_stu3} one can prove relation
$$D_{2,2,1}=\frac{2}{5}E_5+\frac{\zeta(5)}{30}$$
for the graph $\Gamma$ corresponding to $D_{1,1,1}$ but with two of its edges subdivided  by an additional vertex. However, for $D_{1,1,1,1}$ (two-vertex graph with four edges between them) at present, the telescoping method does not yield a proof of the well-known identity
$$D_{1,1,1,1}=24D_{2,1,1}+3E_2^2-18E_4.$$

\subsection{New results} In \cite{kalinin2024mordell} Zagier's formula \eqref{eq_2} was obtained by a telescopic method, along with \eqref{hurwitz}. It then became clear that the telescopic method permits the derivation of the above Hurwitz-type formulas. This telescopic approach allows explicit evaluations of lattice sums by reducing global series to boundary contributions in topographs.


\begin{definition}\label{def:rooted_subtree}
Let $\mathcal T $ be a topograph, and let \(E\) be an oriented edge of $\mathcal T$ labeled \(e_{0}\), with adjacent regions \(r_{0},t_{0}\) (hence $E$ corresponds to the quadratic form $r_0m^2+e_0mn+t_0n^2$).  
Removing \(E\) separates $\mathcal T $ into two  infinite components; 
let $\mathcal T'$ denote the component containing the target $V_0$ of $E$.
The edge \(E\) is called the {\it root} of the subtree $\mathcal T'$. 
We call $\mathcal T'$ {\it the upper half of $\mathcal T$} relative to the root $E$.

\end{definition}

Call an upper half $\mathcal T'$ {\it admissible} if it is an oriented tree, i.e. all edges are oriented ``from the root to infinity'' and $rstefg(V)\ne 0$ for each $V\in \mathcal T'$, i.e. the denominators in the sums below are never zero. The main theorems of this article are as follows.

\begin{theorem} 
\label{thm_5}
Let $\mathcal{T}$ be any topograph of discriminant $D < 0$. Let $\mathcal T'$ be the upper half of $\mathcal T$ with respect to a root $E$. Suppose $\mathcal T'$ is admissible. Then, for a suitable branch of $\arcsin$ and for the principle branch of $\arctan$,
\begin{equation} \label{top1}
 \sum_{\substack{
\topo{T'}}} 
\frac{1}{|rst|} = \frac{1}{D}\left(\frac{e_0}{r_0t_0}-\frac{2\arcsin(\frac{e_0}{r_0t_0}\cdot\frac{\sqrt{-D}}{2})}{\sqrt{-D}}\right),
\end{equation}

\begin{equation} \label{top2}
  \sum_{\substack{
\topoo{\mathcal{T}'}
}} 
\frac{1}{|efg|} =  \frac{1}{D}\left(\frac{\arctan(\frac{\sqrt{-D}}{e_0})}{\sqrt{-D}}-\frac{1}{e_0}\right),  
\end{equation}
where the summation is over all vertices of $\mathcal{T}'$, with each contributing one term. 
\end{theorem}

\begin{theorem} 
\label{thm_6}
Let $\mathcal{T}$ be any topograph of discriminant $D>0$. Let $\mathcal T'$ be the upper half of $\mathcal T$ with respect to a root $E$. Suppose $\mathcal T'$ is admissible. Then
\begin{equation} \label{top3}
 \sum_{\substack{
\topo{T'}
}} 
\frac{1}{|rst|} = \frac{1}{D}\left(\frac{e_0}{r_0t_0}-\frac{2\arcsinh(\frac{e_0}{r_0t_0}\cdot\frac{\sqrt{D}}{2})}{\sqrt{D}}\right),
\end{equation}
\begin{equation} \label{top4}
 \sum_{\substack{
\topoo{\mathcal{T}'}
}} 
\frac{1}{|efg|} =  \frac{1}{D}\left(\frac{\operatorname{arctanh}(\frac{\sqrt{D}}{e_0})}{\sqrt{D}}-\frac{1}{e_0}\right),  
\end{equation}
where the summation is over all vertices of $\mathcal{T}'$, with each contributing one term.
\end{theorem}


We can also pass to the limit when $D\to 0$. We get

\begin{theorem} 
\label{thm_7}
Let $\mathcal{T}$ be any topograph of discriminant $D=0$. Let $\mathcal T'$ be the upper half of $\mathcal T$ with respect to a root $E$. Suppose $\mathcal T'$ is admissible. Then
\begin{equation} \label{top5}
 \sum_{\substack{
\topo{T'}
}} 
\frac{1}{|rst|} =\frac{1}{24}\left(\frac{e_0}{r_0t_0}\right)^3,
\end{equation}
\begin{equation} \label{top6}
 \sum_{\substack{
\topoo{\mathcal{T}'}
}} 
\frac{1}{|efg|} =  \frac{1}{3}\left(\frac{1}{e_0}\right)^3.  
\end{equation}
where the summation is over all vertices of $\mathcal{T}'$, with each contributing one term.
\end{theorem} 
 
\begin{corollary}\label{cor}
As a direct corollary of Theorem 9.12 in \cite{o2024topographs} and formula \eqref{top4}, for non-square $D>0$ we obtain
$$2\log \varepsilon_D = \sum \arctanh\frac{\sqrt{D}}{|e|},$$
where the sum ranges only over edges adjacent to the vertices on the river of the topograph, but not lying in the river (modulo river period); hence the sum is finite. 
\end{corollary}
 
\subsection{Telescoping-over-topograph method of proof}\label{sec_ideas}

We discuss the conceptual picture behind the proof of Theorem~\ref{thm_5}, \eqref{top1}, other statements are similar.

\begin{definition}
Let $\mathcal T''$ be a connected subgraph of $\mathcal T'$, containing $V_0$, such that $V_0$ has degree $2$ in $\mathcal T''$ and all other vertices have degree three or one. See Figure~\ref{fig_2}. An edge $E'\in \mathcal T''$ is called a {\it leaf} if it is not a root and is adjacent to a degree-one vertex of $\mathcal T''$. The set of leaves of $\mathcal T''$  is called the {\it crown} of $\mathcal T''$ (see  Figure~\ref{fig_2}). 
\end{definition}

Fix an exhaustion $\{\mathcal T ''_N\}_{N\ge1}$ of $\mathcal T'$ by finite connected subgraphs, as in Definition~2. Each $\mathcal T ''_N$ contains $V_0$, with crown $C_N$ consisting of the leaves of $\mathcal  T''_N$, and every edge of $C_N$ lies at graph distance $N$ from the root.
At each interior vertex with region labels $(r,s,t)$ and adjacent edge labels $(e,f,g)$ (with appropriate orientations), we use the basic telescoping identity
\begin{equation}\label{eq:vertex-telescope}
\frac{1}{rst}=\frac{1}{D}\!\left(\frac{e}{rt}-\frac{f}{rs}-\frac{g}{st}\right),
\end{equation}
so that, upon summing \eqref{eq:vertex-telescope} over all vertices of $\mathcal T''_N$, the interior contributions in RHS cancel pairwise, leaving (i) a root term and (ii) crown terms supported on $C_N$.

%
%

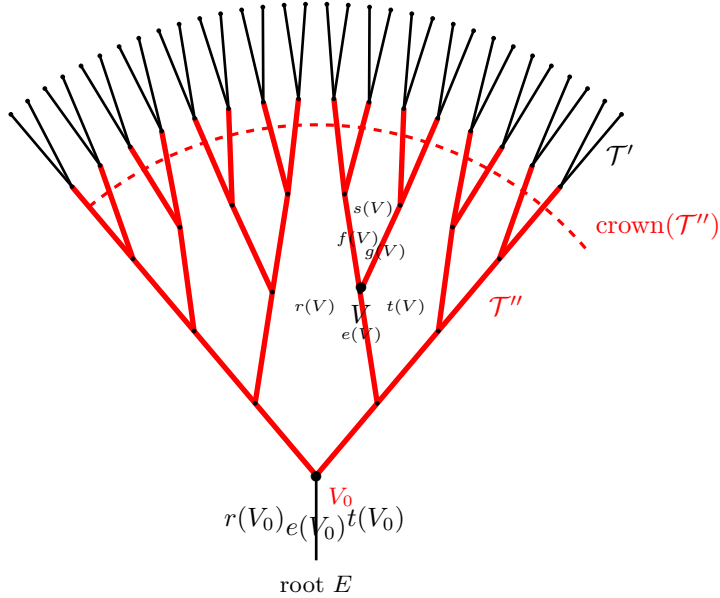
\begin{figure}[h]

\begin{tikzpicture}
\def\maxlevel{5}      
\def\rstep{1.25}
\def\highlight{4}     
\def\angLeft{130}
\def\angRight{50}

\coordinate (v0-0) at (0,0);

\foreach \L in {1,...,\maxlevel}{
  \pgfmathsetmacro{\r}{\L*\rstep}
  \pgfmathtruncatemacro{\NmOne}{int( (2^\L) - 1 )}
  \foreach \i in {0,...,\NmOne}{
    \pgfmathsetmacro{\theta}{\angLeft - (\i*(\angLeft-\angRight)/\NmOne)}
    \coordinate (v\L-\i) at (\theta:\r);
  }
}

\foreach \L in {0,...,\numexpr\maxlevel-1\relax}{
  \pgfmathtruncatemacro{\NmOne}{int( (2^\L) - 1 )}
  \foreach \i in {0,...,\NmOne}{
    \pgfmathtruncatemacro{\Lp}{\L+1}
    \pgfmathtruncatemacro{\cA}{2*\i}
    \pgfmathtruncatemacro{\cB}{2*\i+1}
    \edef\parent{v\L-\i}
    \ifnum\L=0 \edef\parent{v0-0}\fi
    \ifnum\Lp>\highlight
      \draw[line width=1.0pt] (\parent) -- (v\Lp-\cA);
      \draw[line width=1.0pt] (\parent) -- (v\Lp-\cB);
    \else
      \draw[red,line width=2.0pt] (\parent) -- (v\Lp-\cA);
      \draw[red,line width=2.0pt] (\parent) -- (v\Lp-\cB);
    \fi
  }
}

\fill (v0-0) circle (0.9pt);
\foreach \L in {1,...,\maxlevel}{
  \pgfmathtruncatemacro{\NmOne}{int( (2^\L) - 1 )}
  \foreach \i in {0,...,\NmOne}{
    \fill (v\L-\i) circle (0.9pt);
  }
}

\draw[black,line width=1pt] (0,-0.9*\rstep) -- (v0-0);
\node[below left=0.4cm] at (v0-0) {$r(V_0)$};
\node[below right=0.4cm] at (v0-0) {$t(V_0)$};
\node[below =0.4cm] at (v0-0) {$e(V_0)$};

\node[below =1.2cm] at (v0-0) {\small root $E$};
\node[red,below right=0.05cm] at (v0-0) {\small $V_0$};
\node[below=0cm] at (0,0.18) {$\bullet$};

\pgfmathsetmacro{\rc}{(4 - 0.28)*\rstep}
\draw[red,very thick,dashed]
  (\angLeft:\rc) arc[start angle=\angLeft,end angle=\angRight-10,radius=\rc];
\node[red,above] at (4.5,3) {crown($\mathcal T^{\prime\prime}$)};
\node[above] at (4,4) {$\mathcal T^{\prime}$};
\node[red, above] at (2.5,2) {$\mathcal T^{\prime\prime}$};

\coordinate (v1) at (0.59,2.68);
\node[below=0.3cm] at (v1) {$V$};
\node[below=0cm] at (v1) {$\bullet$};

\node[below left=0.3cm] at (v1) {{\tiny $r(V)$}};
\node[above=0.4cm] at (0.75,2.9) {{\tiny $s(V)$}};
\node[below right =0.3cm] at (v1) {{\tiny $t(V)$}};

\node[below left=0.3cm] at (1.2,2.3) {{\tiny $e(V)$}};
\node[above=0.4cm] at (0.55,2.5) {{\tiny $f(V)$}};
\node[below right =0.3cm] at (0.3,3.4) {{\tiny $g(V)$}};

\end{tikzpicture}

\caption{Illustration of $\mathcal T'$, an upper half of $\mathcal T$ relative to a root $E$, a subtree $\mathcal T_4''$ (in red) with a crown $C_4$ consisting of the leaves that intersect the dashed line. For each vertex $V$, we may consider the labels $r(V), s(V), t(V)$ on the adjacent regions and the labels $e(V), f(V), g(V)$ on the adjacent edges.
}\label{fig_2}

\end{figure}

\begin{remark}
Using relations in Section~\ref{sec_trigo}, the identity \eqref{eq:vertex-telescope} can be rewritten as  $fg=e\cdot f+e\cdot g-D$. For $D=-4$, after scaling all variables by $\sqrt{-D}/2$), this is equivalent to the famous identity  \cite{zagier2000quelques} for cotangent function: $$\cot(X)\cot(Y)=\cot(X)\cot(X+Y)+\cot(Y)\cot(X+Y)+1.$$
\end{remark}

To find the limit of the sum of the crown terms when $N\to \infty$, for $D<0$ we use approximation $x\sim \arcsin x+ O(x^3)$ and the identity (see Lemma~\ref{lemma_main})

\[\arcsin(\tfrac{e}{rt}\cdot\tfrac{\sqrt{-D}}{2}) + \arcsin(\tfrac{f}{rs}\cdot\tfrac{\sqrt{-D}}{2})+\arcsin(\tfrac{g}{st}\cdot\tfrac{\sqrt{-D}}{2})=0,\]

Thus, if the following smallness conditions on the crown are satisfied

$$\max_{E\in C_N}\Big|\tfrac{e}{rt}(E)\Big|\longrightarrow 0, \sum_{E\in C_N}\Big|\tfrac{e}{rt}(E)\Big|^{3}\longrightarrow 0 \text{ as } N\to\infty,$$

then $$\sum\limits_{E\in C_N}\tfrac{e}{rt}(E)= \frac{2}{\sqrt{-D}}\arcsin\left(\frac{e_0}{r_0t_0}\cdot\frac{\sqrt{-D}}{2}\right)+ O\left(\sum_{E\in C_N}\Big|\tfrac{e}{rt}\Big|^{3}\right).$$ 

Putting all together and letting $N\to\infty$ proves \eqref{top1}. 

\subsection{Comparison with existing results}
Modular graph functions such as \eqref{eq_string} are defined for positive definite quadratic forms. In \cite{o2024topographs}, the sums of terms of the form $\frac{1}{|rst|}$ are considered for $D<0$ and the sums with summands like $\frac{1}{|efg|}$ for $D>0$. Our approach allows quadratic forms with any discriminant, but one shall restrict summation to admissible halves of a topograph, as the sum over all the vertices of a topograph diverges for $D\geq 0$. We also discuss a geometric interpretation of terms $\frac{1}{rst},\frac{1}{efg}$ that simplify proofs.

One can consider sums with terms of the form $\frac{1}{|r^ns^mt^k|}$ with natural numbers $n,m,k$; in \cite{d2017modular} relations between such sums for various $m,n,k$ were deduced (similar relations are deduced in \cite{duke2021class},\cite{o2024topographs}). Accordingly, one can write the same relations using identities from Section~\ref{sec_trigo}.



To the best of the author's knowledge, the formula for $\log \varepsilon_D$ as in Corollary~\ref{cor} has not previously appeared in the literature. 

It would be desirable to interpret our formulas as evaluations of Eisenstein/Sczech cocycles on modular symbols, in the spirit of \cite{sczech1992eisenstein}.

\subsection{Plan of the paper}

Section~\ref{sec_topographs}  defines topographs and establishes several useful identities for the labels near the vertices of topographs.  Section~\ref{sec_hurwitz} illustrates the main idea by deriving \eqref{sum_1} via telescoping. In Section~\ref{sec_mt} we demonstrate how to obtain Mordell-Tornheim series using our approach. In Section~\ref{sec_proof1} we prove our Theorem~\ref{thm_5} for the case $D<0$ and terms $\frac{1}{rst}$. Section~\ref{sec_duality}  highlights the geometric meaning of the summands and illustrates the duality between formulas involving region labels and those involving edge labels. Here we prove Theorem~\ref{thm_5} for the terms $\frac{1}{efg}$.

 In Section~\ref{sec_hata}, we show that Hata's series for Euler's constant is a particular case of our construction. In Section~\ref{sec_proofs} we prove Theorems~\ref{thm_6},~\ref{thm_7}. Section~\ref{sec:non-square-proof} contains computations for Corollary~\ref{cor} (indefinite case, non-square discriminant), Section~\ref{subsec:square-D} provides similar formulas for the square discriminant. Section~\ref{sec_proofs} presents proofs in the general case.


%

%
%
%
%
%
%
%
%
%
%
%
%
%
%
%
%
%

\section{Topographs}
\label{sec_topographs}
A topograph is a planar, connected, trivalent tree $\mathcal T$ with labels on vertices, edges, and regions (the connected components of $\RR^2\setminus \mathcal{T}$). These labels encode the values of a binary quadratic form $q$ and, at the same time, provide a framework for navigating the set of forms $SL(2,\ZZ)$-equivalent to $q$. Topographs were introduced by J.-H. Conway in \cite{conway1997sensual} in 1997. They provide a powerful geometric tool for visualizing the behavior of binary quadratic forms and their equivalence classes.

Let us start with the graph structure.

\begin{definition}
A superbase is a triple $v_1,v_2,v_3\in\ZZ^2$ such that $v_1+v_2+v_3=0 $ and $\{v_1,v_2\}$ forms a basis of $\ZZ^2$. We consider superbases up to sign, i.e. triples $\{v_1,v_2,v_3\}$ and $\{-v_1,-v_2,-v_3\}$ are identified. Consider a graph $\mathcal T$  whose vertices represent all superbases. Let edges connect vertices of the form $$\{v_1,v_2,-v_1-v_2\}\text{ and }\{v_1,-v_2,-v_1+v_2\}.$$ Thus, each edge corresponds to the four bases $\{\pm v_1,\pm v_2\}$ of $\ZZ^2$. Each vertex $\{v_1,v_2,v_3\}$ has degree three, with incident edges corresponding to $$\{\pm v_1,\pm v_2\},\{\pm v_1,\pm v_3\},\{\pm v_2,\pm v_3\}.$$

\end{definition}

The graph $\mathcal T$ is connected and can be embedded in $\mathbb{R}^2$ without self-intersections; the resulting planar graph is called a topograph. Each region in the complement to the graph corresponds to a primitive vector $\pm v$, see Figure~\ref{fig_1} for a local picture near a vertex and an edge. 
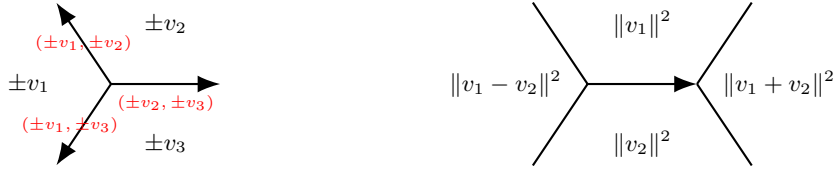
\begin{figure}[h]
\begin{minipage}[b]{0.45\linewidth}
 \begin{tikzpicture}[scale=0.36, every node/.style={font=\Large}, 
    arrow/.style={-{Latex[length=3mm]}, thick},
    redlabel/.style={red, font=\Large}]

  \begin{scope}[shift={(21,5)}]
    \coordinate (B1) at (5,5);
    \coordinate (B2) at (3,8);
    \coordinate (B3) at (3,2);
    \coordinate (B4) at (9,5);

    \draw[arrow] (B1) -- (B2);
    \draw[arrow] (B1) -- (B3);
    \draw[arrow] (B1) -- (B4);

    \node[redlabel] at (7,4.3) {\tiny $(\pm v_2, \pm v_3)$};
    \node[redlabel] at (3.5,3.5) {\tiny$(\pm v_1, \pm v_3)$};
    \node[redlabel] at (4,6.5) {\tiny $(\pm v_1,\pm v_2)$};

    \node at (2,5) {\small $\pm v_1$};
    \node at (7,7.2) {\small $\pm v_2$};
    \node at (7,2.8) {\small $\pm v_3$};
  \end{scope}
\end{tikzpicture}
\end{minipage}
\begin{minipage}[b]{0.45\linewidth}
 \begin{tikzpicture}[scale=0.36, every node/.style={font=\Large}, 
    arrow/.style={-{Latex[length=3mm]}, thick},
    redlabel/.style={red, font=\Large}]

  \begin{scope}[shift={(7,5)}]
    \coordinate (A1) at (3,8);
    \coordinate (A2) at (3,2);
    \coordinate (A3) at (5,5);
    \coordinate (A4) at (9,5);
    \coordinate (A5) at (11,8);
    \coordinate (A6) at (11,2);

    \draw[thick] (A1) -- (A3) -- (A2);
    \draw[arrow] (A3) -- (A4);
    \draw[thick] (A4) -- (A5);
    \draw[thick] (A4) -- (A6);

    \node at (2,5) {\small $\|v_1-v_2\|^2$};
    \node at (7,7.2) {\small$\|v_1\|^2$};
    \node at (7,2.8) {\small$\|v_2\|^2$};
    \node at (12,5) {\small$\|v_1+v_2\|^2$};
  \end{scope}
  \end{tikzpicture}  
 \end{minipage}
  \caption{Left: local picture near a vertex corresponding to a superbase $\{v_1,v_2,v_3\}$. Right:  local picture near an edge, with labels associated to the quadratic form $q(v)=\|v\|^2$.}
  \label{fig_1}
\end{figure}

We label the regions of the topograph with numbers. 
\begin{example}
Label the region corresponding to $\pm v$  by $\|v\|^2 = m^2+n^2$ where $v=(m,n)$. Recall the {\it parallelogram law}: for any $v_1,v_2\in\RR^2$,

$$\|v_1-v_2\|^2+\|v_1+v_2\|^2 = 2(\|v_1\|^2+\|v_2\|^2).$$
\end{example}

Given a binary quadratic form $q$, assign to the region corresponding to $\pm v$ the label $q(v)$.
 
Note that $q(v_1-v_2) + q(v_1+v_2) = 2 \left(q(v_1)+q(v_2)\right)$. Using this identity, one can recover all the values of $q$ on $\ZZ^2$ once  the values of $q$ on a superbase are known.

Conversely, if we label all the regions of the topograph so that, near every edge as in Figure~\ref{apfig}, a), labels satisfy 
\begin{equation}
r+u=2(s+t),
\end{equation}
then these labels determine a unique quadratic form $q$. Indeed, if $$q(m,n)=am^2+bmn+cn^2,$$ then the coefficients $a,b,c$ can be recovered from the identities $$q(\vect{1}{0})=a,q(\vect{0}{1})=c,q(\vect{-1}{-1})=a+b+c.$$

Note that $r, s+t,u$ form an arithmetic progression with difference $g: =s+t-r$, so we label by $g$ the oriented edge pointing from $r$ to $u$, see Figure~\ref{apfig}, a,b). This label changes sign if the orientation of the edge is reversed.

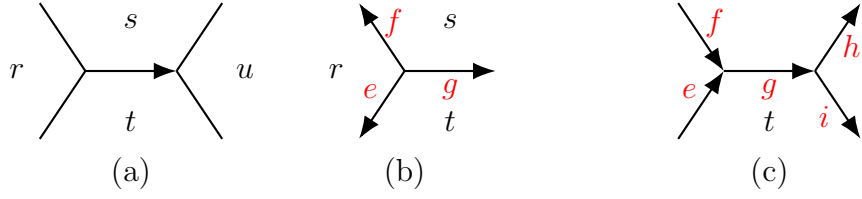
\begin{figure}[h]
\centering
\begin{tikzpicture}[scale=0.3, every node/.style={font=\Large}, 
    arrow/.style={-{Latex[length=3mm]}, thick},
    redlabel/.style={red, font=\Large}]

  \begin{scope}[shift={(7,5)}]
    \coordinate (A1) at (3,8);
    \coordinate (A2) at (3,2);
    \coordinate (A3) at (5,5);
    \coordinate (A4) at (9,5);
    \coordinate (A5) at (11,8);
    \coordinate (A6) at (11,2);

    \draw[thick] (A1) -- (A3) -- (A2);
    \draw[arrow] (A3) -- (A4);
    \draw[thick] (A4) -- (A5);
    \draw[thick] (A4) -- (A6);

    \node at (2,5) {$r$};
    \node at (7,7.2) {$s$};
    \node at (7,2.8) {$t$};
    \node at (12,5) {$u$};
    \node at (7,0.5) {(a)};
  \end{scope}

  \begin{scope}[shift={(21,5)}]
    \coordinate (B1) at (5,5);
    \coordinate (B2) at (3,8);
    \coordinate (B3) at (3,2);
    \coordinate (B4) at (9,5);

    \draw[arrow] (B1) -- (B2);
    \draw[arrow] (B1) -- (B3);
    \draw[arrow] (B1) -- (B4);

    \node[redlabel] at (7,4.2) {$g$};
    \node[redlabel] at (3.5,4.1) {$e$};
    \node[redlabel] at (4.5,7.1) {$f$};

    \node at (2,5) {$r$};
    \node at (7,7.2) {$s$};
    \node at (7,2.8) {$t$};
    \node at (5,0.5) {(b)};
  \end{scope}

  \begin{scope}[shift={(35,5)}]
    \coordinate (C1) at (3,8);
    \coordinate (C2) at (3,2);
    \coordinate (C3) at (5,5);
    \coordinate (C4) at (9,5);
    \coordinate (C5) at (11,8);
    \coordinate (C6) at (11,2);

    \draw[arrow] (C1) -- (C3);
    \draw[arrow] (C2) -- (C3);
    \draw[arrow] (C3) -- (C4);
    \draw[arrow] (C4) -- (C5);
    \draw[arrow] (C4) -- (C6);

    \node[redlabel] at (7,4.2) {$g$};
    \node[redlabel] at (3.5,4.1) {$e$};
    \node[redlabel] at (4.6,7.1) {$f$};
     \node at (7,2.8) {$t$};
    \node[redlabel] at (10.6,6.1) {$h$};
    \node[redlabel] at (9.4,3) {$i$};

    \node at (7,0.5) {(c)};
  \end{scope}

\end{tikzpicture}
\caption{Topographs locally.}
\label{apfig}
\end{figure}

 Similarly, define $e=r+t-s$ and  $f=r+s-t$ (see Figure~\ref{apfig}, b)). Then $e+g=2t$. 
 
 \begin{definition}
 The number {$D:=-ef-fg-eg$}, where $e,f,g$ are the oriented edge labels near a vertex(as in Figure~\ref{apfig} b)), is called the {\bf discriminant} of the topograph.
 \end{definition}
It is straightforward to verify that $D$ is independent of the choice of vertex.
Indeed, in Figure~\ref{apfig} c) (note the different edge orientations!) we have $g-e=2t=i-g$, and hence $$(g-e)(g-f)=(i-g)(h-g)=g^2-eg-gf+ef=g^2+ih-ig-hg,$$
thus $-eg-gf+ef = ih-ig-hg=D$. 

Given a quadratic form $q$ with an associated bilinear form $B(x,y)$, the label on the edge $\{\pm e_i, \pm e_j\}$ is $\pm 2B(\pm e_i, \pm e_j)$, with the sign determined by the orientation of the edge. For example, for the standard dot product $B(x,y)=x\cdot y$ we obtain

$$s+t-r=\|x\|^2+\|y\|^2-\|x-y\|^2 = 2(x\cdot y).$$

\subsection{Useful identities in a topograph}
\label{sec_trigo}
In this section, we state the identities that form the basis of the telescoping and cancellation argument, since they demonstrate that certain first-order approximations are additive, making the telescoping argument straightforward.  
Let us recall the following identity: 

\begin{equation}
\label{eq_stu} 
 \frac{g}{st}+\frac{f}{rs}+\frac{e}{rt}=\frac{gr+ft+es}{rst}=\frac{g(f+e)+f(e+g)+e(f+g)}{2rst}=\frac{-D}{rst}
\end{equation}
\begin{equation}\label{eq_reci}
\frac{1}{e}+\frac{1}{f}+\frac{1}{g}=\frac{ef+fg+ge}{efg}=\frac{-D}{efg}.
\end{equation}
\begin{equation}
\label{eq_stu2} 
 \frac{s}{fg}+\frac{r}{ef}+\frac{t}{eg}=\frac{gr+ft+es}{efg}=\frac{g(f+e)+f(e+g)+e(f+g)}{2efg}=\frac{-D}{efg}
\end{equation}

\begin{equation}
\label{eq_stu3} 
 \frac{g}{s^2t^2}+\frac{f}{r^2s^2}+\frac{e}{r^2t^2}=-\frac{6}{rst}-\frac{D(r+s+t)}{r^2s^2t^2}
\end{equation}

\begin{equation}
\label{eq_stu4} 
 \frac{s}{f^2g^2}+\frac{r}{e^2f^2}+\frac{t}{e^2g^2}=-\frac{3}{2efg}-\frac{D(e+f+g)}{2e^2f^2g^2}
\end{equation}

\begin{equation}
\label{eq_stu5} 
 \frac{1}{e^3}+\frac{1}{f^3}+\frac{1}{g^3}=-\frac{D^3}{e^3f^3g^3}-\frac{3D(e+f+g)}{e^2f^2g^2}+\frac{e}{efg}
\end{equation}

\begin{equation}
\label{eq_stu6} 
 \frac{s}{f^2g^2}+\frac{r}{e^2f^2}+\frac{t}{e^2g^2}=-\frac{3}{2efg}-\frac{D(e+f+g)}{2e^2f^2g^2}
\end{equation}

Interestingly, the combinatorial structure of a topograph also encodes trigonometric and hyperbolic relations, depending on the sign of the discriminant.

\begin{lemma}
\label{lemma_main} For $D<0$, for the appropriate branches of $\arcsin$ and $\arctan$, we have
\begin{equation}\label{eq_arcsin}
\arcsin(\tfrac{e}{rt}\cdot\tfrac{\sqrt{-D}}{2}) + \arcsin(\tfrac{f}{rs}\cdot\tfrac{\sqrt{-D}}{2})+\arcsin(\tfrac{g}{st}\cdot\tfrac{\sqrt{-D}}{2})=0,\end{equation}
\begin{equation}\label{eq_arctan}\arctan(\tfrac{\sqrt{-D}}{e})+\arctan(\tfrac{\sqrt{-D}}{f})+\arctan(\tfrac{\sqrt{-D}}{g})=0.\end{equation}


For $D>0$ and $|e|,|f|,|g|> \sqrt{D}$ we have
\begin{equation}\label{eq_arcsinh}\arcsinh(\tfrac{e}{rt}\cdot\tfrac{\sqrt{D}}{2}) +\arcsinh(\tfrac{f}{rs}\cdot\tfrac{\sqrt{D}}{2})+\arcsinh(\tfrac{g}{st}\cdot\tfrac{\sqrt{D}}{2})=0\end{equation}
\begin{equation}\label{eq_arctanh}\arctanh(\tfrac{\sqrt{D}}{e})+\arctanh(\tfrac{\sqrt{D}}{f})+\arctanh(\tfrac{\sqrt{D}}{g})=0.\end{equation}
\end{lemma}
\begin{proof}
To prove the first identity, note that $-\arcsin A=\arcsin B+\arcsin C$ (for appropriate branches of $\arcsin$) if and only if $-A=B\sqrt{1-C^2}+C\sqrt{1-B^2}$, we then square, and then square again; then all terms cancel. The same argument applies for $\arcsinh$. 
For $\arctan$ and $\arctanh$ these formulas follow from \eqref{eq_reci} and the identities
\[
\arctan(x_1) + \arctan(x_2) + \arctan(x_3) = \arctan\left( \frac{x_1 + x_2 + x_3 - x_1x_2x_3}{1 - x_1x_2 - x_2x_3 - x_3x_1} \right)
\]
\[
\text{arctanh}(x_1) + \text{arctanh}(x_2) + \text{arctanh}(x_3) = \text{arctanh}\left( \frac{x_1 + x_2 + x_3 + x_1x_2x_3}{1 + x_1x_2 + x_2x_3 + x_3x_1} \right)
\]
provided the denominators are nonzero.
\end{proof}

Our arguments rely on the trigonometric identities derived in Section~\ref{sec_trigo}. To proceed, we need the following lemma, in the notation presented in Figure~\ref{apfig2}.
\begin{figure}[h]
\centering
\begin{tikzpicture}[scale=0.36, every node/.style={font=\Large}, 
    arrow/.style={-{Latex[length=3mm]}, thick},
    redlabel/.style={red, font=\Large}]

    \coordinate (B1) at (5,5);
    \coordinate (B2) at (3,8);
    \coordinate (B3) at (3,2);
    \coordinate (B4) at (9,5);

    \draw[arrow] (B1) -- (B2);
    \draw[arrow] (B3) -- (B1);
    \draw[arrow] (B1) -- (B4);

    \node[redlabel] at (7,4.2) {$g$};
    \node[redlabel] at (3.5,4.1) {$e$};
    \node[redlabel] at (4.5,7.1) {$f$};

    \node at (2,5) {$r$};
    \node at (7,7.2) {$s$};
    \node at (7,2.8) {$t$};

\end{tikzpicture}
\caption{Local structure of a topograph near a vertex.}
\label{apfig2}
\end{figure}
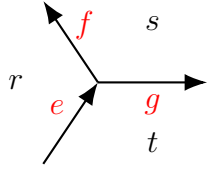

\begin{lemma}
\label{lemma_main2} Consider a topograph with discriminant $D<0$ and its vertex with labels as in  Figure~\ref{apfig2}.  For suitable branches of $\arcsin$ and $\arctan$, we have
$$\arcsin(\frac{e}{rt}\cdot\frac{\sqrt{-D}}{2}) = \arcsin(\frac{f}{rs}\cdot\frac{\sqrt{-D}}{2})+\arcsin(\frac{g}{st}\cdot\frac{\sqrt{-D}}{2})$$

$$\arctan(\frac{\sqrt{-D}}{e})=\arctan(\frac{\sqrt{-D}}{f})+\arctan(\frac{\sqrt{-D}}{g})$$

For a topograph with discriminant $D>0$, we have
$$\arcsinh(\frac{e}{rt}\cdot\frac{\sqrt{D}}{2}) = \arcsinh(\frac{f}{rs}\cdot\frac{\sqrt{D}}{2})+\arcsinh(\frac{g}{st}\cdot\frac{\sqrt{D}}{2})$$

$$\arctanh(\frac{\sqrt{D}}{e})=\arctanh(\frac{\sqrt{D}}{f})+\arctanh(\frac{\sqrt{D}}{g})$$

\end{lemma}
\begin{proof}
Follows immediately from Lemma~\ref{lemma_main} since we only change the orientation of the edge with label $e$.
\end{proof}

\section{Positive-definite case, $D<0$}

This section forms the core of the paper. In Section~\ref{sec_hurwitz} we present a telescopic proof of \eqref{sum_1}, thus highlighting the main idea in the simplest case. In Section~\ref{sec_mt} we apply the telescoping fromulas to the lattice generated by $(1,0)$ and $(1,\mu)$ and obtain Mordell-Tornheim series $$\sum\limits_{n,m\geq 1}\frac{1}{n^2m^2(n+m)^2}=\frac{\zeta(6)}{3}$$ taking the limit $\mu\to 0$. In Section~\ref{sec_duality}, we explain the geometric meaning of the summand $\frac{1}{efg}$ and illustrate the duality between formulas including region labels and formulas including edge labels.

\subsection{Hurwitz series} 
\label{sec_hurwitz}

Let $\det(x, y)$ denote the determinant of the $2\times 2$ matrix with columns $x,y\in\ZZ^2$. Let $$A = \big\{ (x, y) \mid x,y\in \mathbb{Z}_{\geq 0}^2, \det(x, y) = 1 \big\},$$ i.e., the set of pairs of lattice vectors $x=(a,b),y=(c,d)$ in the first quadrant that span lattice parallelograms of oriented area one ($ad-bc=1$).

We prove the Hurwitz formula \eqref{sum_1} via a telescopic method.  
\begin{theorem} 
\label{th_one}
$$4\sum\limits_{A} \frac{1}{\|x\|^2\|y\|^2 \|x+y\|^2} = \pi.$$ 
\end{theorem}

\begin{proof} Define the function $F(x,y) = \frac{2x\cdot y}{\|x\|^2 \|y\|^2}$, $F: (\ZZ^2)^2\to \RR$. An explicit computation shows that
\begin{equation}
\label{eq_1}
F(x,y)-F(x+y,y)-F(x,x+y)=\frac{-4 \det(x,y)^2 }{\|x\|^2\|y\|^2\|x+y\|^2}.
\end{equation}

Consider the sum of these expressions 
over the set $$\{(x,y)| x,y\in \ZZ_{\geq 0}^2\cap [0,n]^2, \det(x, y) = 1\}.$$ 
After cancelling terms with opposite signs, only $F(\vect{1}{0},\vect{0}{1})$ remains,  together with the sum of  $-F(x+y,y)-F(x,x+y)$ over the set $$\left\{(x,y)| x,y\in \ZZ_{\geq 0}^2\cap [0,n]^2, \det(x, y) = 1, x+y\notin \ZZ_{\geq 0}^2\cap [0,n]^2\right\}.$$

Note that each element $(x,y)$ of the above set represents a parallelogram spanned by $x$ and $y$. All these parallelograms have area $1$, and the angles at the origin formed by their sides partition the right angle $\pi/2$ of the first quadrant. Next, $\frac{2x\cdot y}{|x|^2\cdot |y|^2}$ is $2\alpha$  up to third-order terms,  where  $\alpha$ is the angle between $x$ and $y$. Indeed, 
$\sin \alpha\cdot\|x\| \|y\|=1$ and

$$\frac{2x\cdot y}{\|x\|^2 \|y\|^2} = 2\cos{\alpha}\sin\alpha = \sin 2\alpha =  2\alpha -\frac{(2\alpha)^3}{3!}+\dots$$

Also, for $\sum_{i=1}^n\alpha_i=\pi/2$ we have that $$|2\sum_{i=1}^n\sin\alpha_i\cos\alpha_i - \pi |\leq \frac{8\pi}{3}\max_{i=1..N}|\alpha_i|^2\to 0 \text{ as } \max_{i=1..N}|\alpha_i|\to 0.$$
Thus, since $F(\vect{1}{0},\vect{0}{1})=0$, as $n\to \infty$, we obtain

\[F(\vect{1}{0},\vect{0}{1}) + \sum (-F(x+y,y)-F(x,x+y))\to -\pi.\]

Thus, the sum of \eqref{eq_1} considered above yields the desired formula.
\end{proof}

\subsection{Mordell-Tornheim series}
\label{sec_mt}
Fix  $\mu>0$ and consider the sum
\[
\Sigma(\mu)=\sum_{\substack{a \geq b \\ c \geq d \\ ad - bc = 1}} \frac{4\mu^2}{(a^2 + \mu^2 b^2)(c^2 + \mu^2 d^2)\left((a + c)^2 + \mu^2(b + d)^2\right)}.
\]

Equivalently, we may write
\[
\Sigma(\mu)=\sum \frac{4|\det(x,y)|^2}{\|x\|^2 \|y\|^2 \|x + y\|^2},
\]
where 
\[
x, y \in \ZZ_{\geq 0}\cdot\vect{1}{0}  + \ZZ_{\geq 0}\cdot\vect{1}{\mu},
\]

Since $\det(\vect{1}{0},\vect{1}{\mu})=\mu$ (cf. \eqref{eq_1}), this sum can be written via topographs as

\[
\frac{1}{4\mu^2}\Sigma(\mu) =  \sum_{\substack{
\topo{T'}}} 
\frac{1}{|rst|},
\]
where $\mathcal T$ is the topograph for the quadratic form $q(m,n)=(m+n)^2+\mu^2n^2$ and the root $E$ of $\mathcal T'$ corresponds to the edge with $$r_0=q(\vect{1}{0}), t_0=q(\vect{0}{1}), e_0 = q(\vect{1}{1})-q(\vect{1}{1})-q(\vect{0}{1}).$$

Applying the telescoping method with the same function $F$ as in the proof of Theorem~\ref{th_one}, we obtain
\[
\Sigma(\mu)= - \left( F\left(\vect{1}{0},\vect{1}{\mu}\right) - \lim\limits_{N\to\infty}\sum_{E\in C_N} F(E_u, E_v) \right),
\]
where for $E\in C_N$ in the crown, with adjacent regions corresponding to vectors $(a,b),(c,d)$ we define 
$$E_u = a\cdot\vect{1}{0}  + b\cdot\vect{1}{\mu}, E_v=c\cdot\vect{1}{0}  + d\cdot\vect{1}{\mu}.$$

Substituting $F\left(\vect{1}{0},\vect{1}{\mu}\right) = \frac{2}{1+\mu^2} = \frac{e_0}{r_0t_0}$ into the formula above and using that the area of parallelogramms is $\mu$ and they partition the angle between $\vect{1}{0}, \vect{1}{\mu}$, which is $\arctan \mu$, we obtain

\begin{theorem}
\[
\sum_{\substack{a \geq b \\ c \geq d \\ ad - bc = 1}} \frac{1}{(a^2 + \mu^2 b^2)(c^2 + \mu^2 d^2)\left((a + c)^2 + \mu^2(b + d)^2\right)} =  \frac{1}{4\mu^2}\left( \frac{2\arctan \mu}{\mu} - \frac{2}{1 + \mu^2} \right).
\]
\end{theorem}
Taking the limit as \( \mu \to 0 \) gives $\Sigma(\mu)/{4\mu^2}\to 1/3$.
Indeed,
\begin{equation}
\label{eq_mordell}
\sum_{\substack{a,c\geq 1\\ \gcd(a,c)=1}} \frac{1}{a^2 c^2 (a + c)^2} = \frac{1}{3},
\end{equation}
because for each such a pair $(a,c)$ there exists a unique pair $(b,d)\in\ZZ_{\geq 0}^2$ with $ad-bc=1, b\leq a, d\leq c$. This identity \eqref{eq_mordell} is well-known \cite{tornheim1950harmonic, hata1995farey}.

Since on both sides of equality we have analytic functions, we may substitute \( \mu = \frac{i}{2} \). Then,
\[
\sum_{\substack{a \geq b \\ c \geq d \\ ad - bc = 1}} \frac{1}{(4a^2 - b^2)(4c^2 - d^2)\left(4(a + c)^2 - (b + d)^2\right)}= 
\]
\[
= \frac{1}{64} \cdot \left( \frac{\arctan \tfrac{i}{2}}{\tfrac{i}{2}} - \frac{1}{1 + (\tfrac{i}{2})^2} \right) \bigg/ \left(2 \cdot \left( \tfrac{i}{2} \right)^2\right)=
 \frac{\frac{4}{3} - \ln 3}{32}\ \ .
\]

%
%
%
%
%
%
%
%
%
%
%
%


\begin{remark}
Since
$$\sum_{\substack{a \geq b \\ c \geq d \\ ad - bc = 1}} \frac{1}{(a^2 + \mu^2 b^2)(c^2 + \mu^2 d^2)\left((a + c)^2 + \mu^2(b + d)^2\right)}=\frac{1}{2\mu^2} \left( \frac{\arctan \mu}{\mu} - \frac{1}{1 + \mu^2} \right) =  $$

$$=\sum_{k=0}^\infty (-1)^k\frac{k+1}{2k+3}\mu^{2k} = \frac{1}{3}-\frac{2}{5}\mu^2+\frac{3}{7}\mu^4-\dots.$$
\end{remark}

Thus, by comparing coefficients of $\mu^{2k}$ one can derive further identities. For example, examining the coefficient of $\mu^2$ we obtain
$$\sum_{\substack{a \geq b \\ c \geq d \\ ad - bc = 1}}\frac{1}{a^2c^2(a+c)^2}\left(\frac{b^2}{a^2}+\frac{d^2}{c^2}+\frac{(b+d)^2}{(a+c)^2}\right) = \frac{2}{5}\ .$$

\subsection{Proof of Theorem~\ref{thm_5}, first part} \label{sec_proof1}

Let us interpret Theorem~\ref{th_one} in terms of topographs. Indeed, it is enough to plug into the folowing formula \eqref{top1} the edge $E=\{\pm \vect{1}{0},\pm\vect{0}{1}\}$ and the positive-definite quadratic form $q(n,m)=n^2+m^2$, with $D=-4$, $e_0=0, r_0=t_0=1$.
\[
 \sum_{\substack{
\topo{T'}}} 
\frac{1}{|rst|} = \frac{1}{D}\left(\frac{e_0}{r_0t_0}-\frac{2\arcsin(\frac{e_0}{r_0t_0}\cdot\frac{\sqrt{-D}}{2})}{\sqrt{-D}}\right),
\]

To prove Theorem~\ref{th_one}, we performed telescoping of $$\frac{e}{rt}=2\frac{x\cdot y}{|x|^2 |y|^2}.$$

The $\frac{1}{D}=-\frac{1}{4}$ corresponds to the factor $-4$ in \eqref{eq_1}. The only surviving term $\frac{e_0}{r_0t_0}$ is zero in this case.

Each crown term, up to third order, was $$\frac{2\arcsin(\frac{e_0}{r_0t_0}\cdot\frac{\sqrt{-D}}{2})}{\sqrt{-D}} = \arcsin(\frac{e_0}{r_0t_0}).$$

So the geometric meaning of the terms that we telescope is the angle between vectors. Note that in order to obtain the telescoping relation, one may consider quadratic forms other  than $q(v) = \|v\|^2$. 

\begin{proof}[Proof of Theorem~\ref{thm_5}]

Without loss of generality we may assume that the root edge $E$ corresponds to $\{\pm e_1,\pm e_2\}$, $e_1=\vect{1}{0}, e_2=\vect{0}{1}$. Call by $q'$ the positive definite quadratic form used to construct the topograph.

Recall that $q'$ is $SL(2,\RR)$-equivalent to $\mu \cdot q$ for  $q(v)=\|v\|^2$ and $\mu =\frac{\sqrt{-D}}{2}$. Under this equivalence the vectors $e_1,e_2$ go to vectors $e_1',e_2'$ such that $$\mu q(e_i')=q'(e_i), i=1,2.$$ Values of $r,s,t$ near $\mathcal T'$ for the topograph of $q'$ are the values of $\mu q$ on vectors $me_1'+ne_2'$ for $m,n\in\ZZ_{\geq 0}, \gcd(m,n)=1$.

 Now using the same $F$ as in the proof of Theorem~\ref{th_one}, we obtain

\[ \sum_{\substack{
\topo{T'}}} 
\frac{1}{|rst|} = \frac{1}{\mu^3}\frac{-1}{4}\left(F(e_1',e_2') - 2(\text{the\ angle\ between } e_1',e_2')\right) =\]
\[=\frac{-1}{4\mu^2} \left(\frac{e_0}{r_0t_0}-\frac{\arcsin(\mu\frac{e_0}{r_0t_0})}{\mu} \right)= \frac{1}{D}\left(\frac{e_0}{r_0t_0}-\frac{2\arcsin(\frac{e_0}{r_0t_0}\cdot\frac{\sqrt{-D}}{2})}{\sqrt{-D}}\right)\]

Note that in order to write the identity $$2(\text{the\  angle\ between } e_1',e_2') = \arcsin(\mu\frac{e_0}{r_0t_0})=\arcsin(\tfrac{e_0}{r_0t_0}\cdot\tfrac{\sqrt{-D}}{2})$$ we have to choose the suitable  brach of arcsine as the angle between $e_1',e_2'$ may be bigger than $\pi/2$.
\end{proof}

Another (and somewhat more transparent) way to prove this theorem might be to sum the identity ~\eqref{eq_stu}

$$\frac{g}{st}+\frac{f}{rs}+\frac{e}{rt}=\frac{gr+ft+es}{rst}=\frac{g(f+e)+f(e+g)+e(f+g)}{2rst}=\frac{-D}{rst}$$

over all the vertices of $\mathcal T'$ and then use \eqref{eq_arcsin}
$$\arcsin(\tfrac{e}{rt}\cdot\tfrac{\sqrt{-D}}{2}) + \arcsin(\tfrac{f}{rs}\cdot\tfrac{\sqrt{-D}}{2})+\arcsin(\tfrac{g}{st}\cdot\tfrac{\sqrt{-D}}{2})=0$$
and $$\tfrac{e}{rt}\cdot\tfrac{\sqrt{-D}}{2}=\arcsin(\tfrac{e}{rt}\cdot\tfrac{\sqrt{-D}}{2})+O(\big|\tfrac{e}{rt}\big|^3).$$

Then it remains to prove that the sum of $\big|\tfrac{e}{rt}\big|^3$ over the crown $C_N$ tends to $0$ as $N\to \infty$, and this is hard to witthout referring to the nature of $r,s,t$ and explicit estimates about their growth. Also, the choice of the right branch of arcsine is much less transparent. 

\begin{remark}
Consider the sum of $\frac{1}{|rst|}$ over the complement $\mathcal T\setminus \mathcal T'$ to $\mathcal T'$. Using the above procedure, it evauates as

$$\frac{1}{\mu^3}\frac{-1}{4}\left(F(e_1',-e_2') - 2(\text{the\ angle\ between } e_1',e_2')\right)$$

and the sum over the whole topograph $\mathcal T$ is therefore $\frac{\pi}{2\mu^3} = \frac{4\pi}{|D|^{3/2}}$ as it is predicted by Theorem~\ref{th_three}.

\end{remark}

\subsection{Duality: inside and outside the circle, Theorem~\ref{thm_5}, second part}\label{sec_duality}

The telescopic identities admit a natural geometric interpretation, revealing a duality between sums of products of reciprocals of region labels and sums of the products of reciprocals of edge labels of the topograph. The product of the region labels at a vertex corresponds to the area of the inscribed triangle, whereas the product of the edge labels corresponds to the area of the triangle formed by the tangent lines at the corresponding points (cf. Legendre duality in \cite{certain}).

To relate the Farey tessellation to points on the unit circle, we use the rational parametrization of $x^2+y^2=1$ by 
 $$f:\vect{a}{b}\to\left(\frac{2ab}{a^2+b^2}, \frac{a^2-b^2}{a^2+b^2}\right).$$ 
 
Note that $f(\vect{1}{0})=(0,1),f(\vect{0}{1})=(0,-1)$. By a direct calculation, the area of the triangle with vertices  
$f(\vect{a}{b}),f(\vect{c}{d}),f(\vect{a+c}{b+d})$ equals to 
\begin{equation}
\label{eq_3}
\frac{2|ad-bc|^3}{(a^2+b^2)\cdot(c^2+d^2)\cdot ((a+c)^2+(b+d)^2)}=\frac{2}{rst}.
\end{equation}
Here $r,s,t$ denote the values of the quadratic form $q(n,m)$ on a superbase vectors $\left\{\vect{a}{b},\vect{c}{d},-\vect{a+c}{b+d}\right\}$, the notation that we use for the topograph's labels on regions.

\begin{theorem}[\cite{hurwitz1905darstellung}] 
$$\sum\limits_{\substack{a,b,c,d\in\ZZ_{\geq 0}\\ ad-bc=1}} \frac{1}{(a^2+b^2)(c^2+d^2)((a+c)^2+(b+d)^2)} = \frac{\pi}{4}.$$
\end{theorem}
\begin{proof}
The area of the right half of the unit disc is $\pi/2$. The triangles with vertices $$f(\vect{a}{b}),f(\vect{c}{d}),f(\vect{a+c}{b+d}), a,b,c,d\geq 0, ad-bc=1$$ tile it completely. Thus we divide \eqref{eq_3} by two and sum over all tuples $(a,b,c,d)$ with $a,b,c,d\geq 0, ad-bc-1$.
\end{proof} 

This reasoning goes back to the original article of A. Hurwitz \cite{hurwitz1905darstellung}. Hurwitz's proof applies for any positive-definite binary quadratic form $q$ (the above case corresponds to $q(v)=\|v\|^2, v\in\ZZ^2$) and consists of using a rational parametrization of a quadric curve to partition its interior into triangles corresponding to consecutive Farey fractions $(\frac{a}{b},\frac{a+c}{b+d}, \frac{c}{d})$, and then verifying that the areas of these triangles are proportional to $$(q(a,b) q(c,d) q(a+c,b+d))^{-2}.$$

Alternatively, consider the tangent lines $l_{a,b}$ to the unit circle at points $f(\vect{a}{b})$. The area of the triangle formed by these lines $l_{a,b},l_{c,d},l_{a+c,b+d}$ is
$$\frac{|ad-bc|^3}{(ac+bd)(a(a+c)+b(b+d))((a+c)c + (b+d)d)} =\frac{1}{efg}.$$

Note that the dot products in the denominators are exactly $e,f,g$ on the edges in the topograph. Since these triangles tile the region between the unit circle and the tangents at $(0,1)$ and $(1,0)$, we can evaluate the sum $\sum\frac{1}{efg}$, namely:
\begin{theorem}
$$\sum_{\substack{a \geq b \\ c \geq d \\ ad - bc = 1}}\frac{1}{(ac+bd)(a(a+c)+b(b+d))((a+c)c + (b+d)d)}=1-\pi/4. $$
\end{theorem}

\begin{remark}
This identity may be generalized by deforming the lattice as in Section~\ref{sec_mt}.  Consider the vectors $a\vect{1}{0}+b\vect{1}{\mu}$ and draw the tangent lines at points $f(\vect{a}{\mu b})$. Then the following sum is equal to the area of a part between of the unit circle and  tangents to it at $(0,1)$ and $\frac{2\mu}{\mu^2+1}, \frac{\mu^2-1}{\mu^2+1}$, namely

$$\sum_{\substack{a \geq b \\ c \geq d \\ ad - bc = 1}}\frac{\mu^3}{(ac+\mu^2bd)(a(a+c)+\mu^2b(b+d))((a+c)c + \mu^2(b+d)d)}=\mu-\arctan\mu.$$

Dividing by $\mu^3$ and substituting $\mu=0$ we recover  \eqref{eq_mordell} once again:

$$\sum_{\gcd(a,c)=1}\frac{1}{(ac)(a(a+c))((a+c)c)}=\frac{1}{3}.$$

Expanding $\arctan$ into its Taylor series and considering the next coefficient, we obtain

$$\sum_{\substack{a \geq b \\ c \geq d \\ ad - bc = 1}}\frac{1}{a^2c^2(a+c)^2}\left(\frac{bd}{ac}+\frac{b(b+d)}{a(a+c)}+\frac{(b+d)d}{(a+c)c}\right)=\frac{1}{5}.$$

We may also plug $\mu=i/2$, getting

$$\sum_{\substack{a \geq b \\ c \geq d \\ ad - bc = 1}}\frac{1}{(4ac-bd)(4a(a+c)-b(b+d))(4(a+c)c-(b+d)d)}=$$
$$=\frac{1}{64(i/2)^3}(i/2-\arctan{i/2})=\frac{\ln 3-1}{16}.$$

\end{remark}

\begin{proof}[Proof of Theorem~\ref{thm_5}, \eqref{top2}] 
As in Section~\ref{sec_proof1}, apply an $SL(2,\RR)$ transformation to bring the positive definite quadratic form to $\mu(m^2+n^2)$, and then use the geometric arguments above.  Thus $\sum\frac{1}{efg}$ is proportional to the area between the unit circle and  two tangents to it.

The condition $e_0>0$ guaranties that the tangents to the unit circle intersect and that the sum is finite. The arctan term is proportional to the area of the sector, and $\frac{1}{e_0}$ is proportional to the area of the triangle, which we subtract to get the desired area.

\end{proof}

We also can compare the above proof with a direct approach. We can telescope $\frac{s}{fg}$ and use the telescoping identity \eqref{eq_stu2}.

Denote  by $f = f_0, f_1,\dots $ the edges of the region $s$ on the left and by $g=g_0, g_1, \dots$ those on the right. Note that $f_{k+1}=f_{k}+2s,g_{k+1}=g_{k}+2s$ and 
$$\sum_{k=0}^\infty\frac{s}{f_kf_{k+1}}=\frac{s}{2s}\sum_{k=0}^\infty (\frac{1}{f_k}-\frac{1}{f_{k+1}})=\frac{1}{2f_0}.$$

Therefore in the sum of $\frac{s}{fg}$  the term $\frac{1}{2e_0}$ appears twice, all intermediate terms cancel because $\frac{1}{2f_0}+\frac{1}{2g_0}-\frac{s}{f_0g_0}=0$ for $s=\frac{f_0+g_0}{2}$. The sum of the terms at the crown is $-\sum\frac{1}{e}$ over all $e$ in the crown. Then we use the approximation $$\frac{\sqrt{-D}}{e}\sim \arctan(\frac{\sqrt{-D}}{e})$$ at the edges of the crown and  Lemma ~\ref{lemma_main2}

$$\arctan(\frac{\sqrt{-D}}{e})=\arctan(\frac{\sqrt{-D}}{f})+\arctan(\frac{\sqrt{-D}}{g}).$$
However, the direct estimates of the quality of approximations and the right choice of the branch or arctan is not transparent, thereby we presented the above geometric proof.
\section{Indefinite case, $D>0$}

In Section~\ref{sec_hata} we derive Hata's series for the Euler constant, using a form $q(n,m)=nm$.

\subsection{Hata's series}
\label{sec_hata}
We now examine a series for Euler's constant $\gamma$, due to Masayoshi Hata.

\begin{definition} Let $\mathcal{F}$ denote the set of ordered pairs of reduced fractions $\left(\frac{a}{b}, \frac{c}{d}\right)$ with
$0\leq \frac{a}{b} < \frac{c}{d}\leq 1$,
$ad - bc= -1$.
Thus $\mathcal{F}$ is the set of pairs of consecutive Farey fractions. Let $\mathcal{F}^* = \left\{\left(\frac{a}{b},\frac{c}{d}\right) = (\frac{0}{1},\frac{1}{n}): n \in \mathbb{N}\right\}$. 
\end{definition}
Masayoshi Hata proved the following theorem.
\begin{theorem}[\cite{hata1995farey}]
\label{thm}
In the above notation
\[
\gamma = \frac{1}{2} + \frac{1}{2} \sum_{\left(\frac{a}{b}, \frac{c}{d}\right) \in \mathcal{F} \setminus \mathcal{F}^*} \frac{1}{abcd(a+c)(b+d)}.
\]
\end{theorem}

Hata's proof proceeds by analysing  presentations of functions in a Schauder basis associated with pairs of consecutive Farey fractions. Using a Parseval-type identity for the function $\psi(x)=x\{\frac{1}{x}\}(1-\{\frac{1}{x}\})$, he thereby obtains the above theorem.

Note that $$\frac{1}{abcd(a+c)(b+d)}=\frac{1}{q(x)q(y)q(x+y)}=\frac{1}{rst}$$ for the quadratic form $q(v)= mn$, for  $v=(m,n)$. Thus, this identity may also be derived via our telescopic method.

\begin{lemma}
For $a,b,c,d\geq 0$ with $ad-bc=-1$ one has
$$\arcsinh\left(\frac{ad+bc}{2abcd}\right)=\log\left(\frac{bc}{ad}\right).$$
\end{lemma}
\begin{proof}

By a direct check, we see that 
$$1+\left(\frac{(ad+bc)}{2abcd}\right)^2 = \left(\frac{(ad+bc)^2+1}{4abcd}\right)^2,$$

then,

$$\arcsinh\left(\frac{ad+bc}{2abcd}\right) = \log\left(\frac{(ad+bc)}{2abcd}+\sqrt{1+\left(\frac{(ad+bc)}{2abcd}\right)^2}\right)=$$

$$\log\left(\frac{(ad+bc)^2+2(ad+bc)+1}{4abcd}\right)=\log\left(\frac{(bc)^2}{abcd}\right)=\log\left(\frac{bc}{ad}\right).$$

\end{proof}

Thus, we we obtain the relation 

$$\arcsinh\left(\frac{ad+bc}{2abcd}\right) = \arcsinh\left(\frac{a(b+d)+b(a+c}{2ab(a+c)(b+d)}\right)+\arcsinh\left(\frac{(a+c)d+(b+d)c}{2(a+c)(b+d)cd}\right).$$

In order to find $\sum\frac{1}{rst}$ one takes the telescopic sum of $\frac{e}{rt}-\frac{g}{st}-\frac{h}{rs}$ which, in our case, becomes (since $e = q(x+y)-q(x)-q(y)$)

$$F(\vect{a}{b},\vect{c}{d})=\frac{ad+bc}{abcd}\sim 2\arcsinh\left(\frac{ad+bc}{2abcd}\right),$$
if it is small.

Therefore

$$\sum_{\left(\frac{a}{b}, \frac{c}{d}\right) \in \mathcal{F} \setminus \mathcal{F}^*} \frac{1}{abcd(a+c)(b+d)} =\sum_{n=1}^\infty (F(\vect{1}{n},\vect{1}{n+1})-2\arcsinh\left(\frac{1\cdot(n+1)+n\cdot 1}{2n(n+1)}\right))=$$

$$=2\lim_{n\to\infty} (1+\frac{1}{2}+\frac{1}{3}+\dots+\frac{1}{n}-\log n)-1 = 2\gamma-1.$$

Using the topograph relation \eqref{eq_stu3} and the same strategy of proof we obtain

\begin{theorem}
In the above notation
\[
\sum_{\left(\frac{a}{b}, \frac{c}{d}\right) \in \mathcal{F} \setminus \mathcal{F}^*} \frac{ab+cd+(a+c)(b+d)}{(abcd(a+c)(b+d))^2} = 7-12\gamma.
\]
\end{theorem}

\subsection{Proofs of Theorem~\ref{thm_6} and Theorem~\ref{thm_7}}\label{sec_proofs}

We consider the case of a non-square discriminant $D>0$. Throughout, $\mathcal T$ denotes a topograph of discriminant $D$, with edge labels $(e,f,g)$ and adjacent region labels $(r,s,t)$ at a vertex as in Section~\ref{sec_topographs}.

For $D>0$, there is a unique periodic bi-infinite \emph{river} (a bi-infinite path) on $\mathcal T$ which separates regions with labels of opposite signs.

If $D=m^2$ then there are two {\it lakes} in the topograph, i.e. two regions with label $0$. If $D=0$ then there is exactly one lake; see details about rivers and lakes in \cite{conway1997sensual, o2024topographs, o2024topographs}.  Note that in the case $D\geq 0$ the values $\big|\frac{e}{rt}\big|,\big|\frac{1}{e}\big|$ on crowns $C_N$ tend to zero as $N\to\infty$ if we consider $\mathcal T'$ which does not intersect river and lakes, and crowns $C_N\subset \mathcal T''\subset \mathcal T$. Thus Theorem~\ref{thm_6} follows from using methods of Section~\ref{sec_ideas} and telescoping identities \eqref{eq_stu}, \eqref{eq_stu2} and Lemma~\ref{lemma_main}.

Than the second part of Theorem~\ref{thm_7}  is obtained by telescoping identities \eqref{eq_stu5} (for $D=0$).
The first part of Theorem~\ref{thm_7} (actually, as well as the first part) is equivalent (up to multiplication by a constant) to well-known identity 

$$\sum_{(m,n)=1}\frac{1}{m^2n^2(m+n)^2}=\frac{1}{24}$$

Theorem~\ref{thm_7} is also can be deduced from the case $D>0$ by taking the limit $D\to 0$.


\subsection{Non-square positive discriminants: proof of Corollary~\ref{cor}}\label{sec:non-square-proof}

 Following O’Sullivan, define the modified topograph $\mathcal T^\star$ by keeping $\mathcal T$ unchanged except that every \emph{river edge} is relabeled by $\sqrt{D}$; the region labels off the river remain unchanged. One then sums \emph{edge} contributions over $\mathcal T^\star$ modulo the river’s fundamental period. In O’Sullivan’s formulation this yields the class number series for $D>0$ in terms of edge labels \cite[Thm.~9.2]{o2024topographs}: 
\[
D^{3/2}\!\!\sum_{(e,f,g)\in \mathcal T^\star/\text{period}}\frac{1}{|efg|}\;=\;2\log\varepsilon_D,
\]
with the sum taken over one river period of $\mathcal T^\star$ (one term per vertex). 

Now, take any edge $E$ with a source on the river and the target outside of the river, as a root for $\mathcal T'$, a half of $\mathcal T$. As stated in \eqref{top4} 
\begin{equation*} 
 \sum_{\substack{
\topoo{\mathcal{T}'}
}} 
\frac{1}{|efg|} =  \frac{1}{D}\left(\frac{\operatorname{arctanh}(\frac{\sqrt{D}}{e})}{\sqrt{D}}-\frac{1}{e}\right),  
\end{equation*}
where $e$ is the label on $E$.

Summing over all the vertices in the river period (recall that the labels on the river edges are $\sqrt{D}$)  and all such edges we obtain 

{\bf Corollary~\ref{cor}.}
\[
2\log\varepsilon_D=D^{3/2}\!\!\sum_{(e,f,g)\in \mathcal T^\star/\text{period}}\frac{1}{|efg|}=\]
\[=\sum_{e\in \text{river}/\text{period}} \frac{1}{e} + \sum_{e\in \text{river}/\text{period}} \left(\arctanh\frac{\sqrt{D}}{|e|}-\frac{1}{e}\right) = \sum_{e\in \text{river}/\text{period}} \arctanh\frac{\sqrt{D}}{|e|},
\]
where the sum ranges only over edges $E$ with labels $e$, adjacent to the vertices on the topograph's river, but not belonging to the river, modulo river period, hence the sum is finite. 

\begin{example}
For $q(x,y)=x^2-2y^2$ we obtain
$$
2\log (3+2\sqrt2) =2\log \varepsilon_8 = 4\arctanh\frac{2\sqrt2}{4}=4\frac{1}{2}\log\frac{1+\frac{1}{\sqrt2}}{1-\frac{1}{\sqrt2}}= 2\log (3+2\sqrt2).
$$
\end{example}

\subsection{Class number for square discriminants}\label{subsec:square-D}

In the case of a \emph{square} discriminant $D=m^{2}$,  the usual topograph features a \emph{river} between two regions ({\it lakes}) with zero labels. 

Define an auxiliary function
\[
W_1(x) = 2\int_0^\infty \Re\left( \frac{y}{(y^2+1)(e^{\pi(y+2ix)} - 1)} \right) dy.
\]
It is immediate that $W_1(x)$ is $1$-periodic and even.
\begin{theorem}[\cite{digamma}]\label{thm:main}
For each real $x$ with $\tfrac12\pm x\notin\{0,-1,-2,\dots\}$, 
\begin{equation}\label{eq:main-id}
2W_1(x) + \log 4 + \psi\left(\tfrac{1}{2} + x\right) + \psi\left(\tfrac{3}{2} - x\right) = 0.
\end{equation}  
\end{theorem}

In particular, \eqref{eq:main-id} holds for all rational numbers \(x = r/m \in (0,1)\).

\begin{theorem}\cite[Thm.~9.4]{o2024topographs}\label{thm:Osull-94} We refer to the notation of \cite{o2024topographs}.
Let $\mathcal T$ be any topograph of square discriminant $D=m^{2}>1$. Form the modified topograph $\mathcal T^{\star}$ by
keeping $\mathcal T$ unchanged except that every river edge directed rightwards is relabeled by $\sqrt{D}=m$.
Let $r$ and $s$ be the congruence classes modulo $m$ of the region labels adjacent to a lake. Then
\[
m^{3}\!\!\sum_{\substack{(e,f,g)\in T^{\star}\\ \text{vertex not on a lake}}}\frac{1}{|efg|}
\;+\;W_{1}\!\left(\frac{r}{m}\right)\;+\;W_{1}\!\left(\frac{s}{m}\right)
\;=\;2\log\!\left(\frac{m}{2\,\gcd(m,r)}\right).
\]
(Each vertex contributes one term in the sum.) For $D=m=1$, the identity remains valid after adding $2$
to the right-hand side.
\end{theorem}

Let $\gcd(m,r)=1$. Near the lake with nearby label equal to $r \pmod m$, the labels on edges are

$$e_i=2r+(2k+1)m, k=0,1,\dots, l_i =  2r-(2k+1)m, k=1,2,\dots$$

Thus, using \eqref{top4} we evaluate the sum of $\frac{1}{|efg|}$ as

$$\sum (\arctanh\frac{m}{|e_i|}-\frac{m}{e_i})+\sum (\arctanh\frac{m}{|l_i|}-\frac{m}{l_i}).$$

Thus, after simplification, we obtain

$$2\log\!\left(\frac{m}{2\,\gcd(m,r)}\right)-(W_{1}\!\left(\frac{r}{m}\right)\;+\;W_{1}\!\left(\frac{s}{m}\right)) =m^{3}\!\!\sum_{\substack{(e,f,g)\in T^{\star}\\ \text{vertex not on a lake}}}\frac{1}{|efg|}=$$
$$=\sum_{\text{river}} \arctanh\frac{m}{|e|} -\frac{1}{2}\left(\log\frac{r(m-r)}{m^2}+\psi(\frac{r}{m}+\frac{1}{2})+\psi(\frac{3}{2}-\frac{r}{m})\right)-$$ 
$$-\frac{1}{2}\left(\log\frac{s(m-s)}{m^2}+\psi(\frac{s}{m}+\frac{1}{2})+\psi(\frac{3}{2}-\frac{s}{m})\right).
$$

Using \eqref{eq:main-id} we obtain
\begin{corollary}
$$\sum_{e\in \text{river}} \arctanh\frac{m}{|e|} = \frac{1}{2}\log(r(m-r)s(m-s)).$$
\end{corollary}
\begin{example}
For $m=3$ we get only one edge from the river, and, indeed, $\arctanh{\frac{3}{5}}=\frac{1}{2}\log 4.$
\end{example}

Then, recall the class-number identity (9.29)\cite{o2024topographs} (for \(m>1\))
\begin{equation}\label{eq:929}
h(m^2)\,\log\!\Big(\frac{m}{2}\Big)
= \end{equation}
$$=\sum_{\substack{b^2-4ac=m^2\\ \gcd(a,b,c)=1\\ a,c,\, a+b+c>0}}\frac{m^3}{3b(b+2a)(b+2c)}
\;+\; \sum_{\substack{[a,b,c]\ Z\text{-reduced}\\ b^2-4ac=m^2\\ \gcd(a,b,c)=1}}\frac{m}{b}
\;+\; \sum_{\substack{1\le r<m\\ (r,m)=1}} W_1\!\left(\frac{r}{m}\right).
$$
Here the first sum ranges over primitive triples \([a,b,c]\in\ZZ^3\) of discriminant \(m^2\) with the sign condition \(a,c,a+b+c>0\);
the second sum runs over primitive \(\ZZ\)-reduced representatives \([a,b,c]\) (in the paper's sense of reduction).

From \eqref{eq:main-id}  and the Gauss multiplication theorem for \(\psi\), one finds in \cite{digamma} that
\begin{equation}\label{eq:W1-sum}
\sum_{\substack{1\le r<m\\ (r,m)=1}} W_1\!\left(\frac{r}{m}\right)
\;=\; \varphi(m)\log\!\Big(\frac{m}{2}\Big)
\;+\; \varphi(m)\sum_{p\mid m}\frac{\log p}{p-1}
\;-\; m\sum_{d\mid m}\frac{\mu(d)}{d}\,\psi\!\left(\frac{m}{2d}\right).
\end{equation}

It holds that \(h(m^2)=\varphi(m)\) for \(m>1\).
Substituting \eqref{eq:W1-sum} into \eqref{eq:929} and replacing \(h(m^2)\) by \(\varphi(m)\), the terms
\(\varphi(m)\log\!\bigl(\tfrac{m}{2}\bigr)\) cancel and we obtain:

\begin{corollary}\label{prop:main}
For every integer \(m>1\),
\begin{equation*}\label{eq:final-identity}
\sum_{\substack{b^2-4ac=m^2\\ \gcd(a,b,c)=1\\ a,c,\, a+b+c>0}}\frac{m^3}{3b(b+2a)(b+2c)}
\;+\; \sum_{\substack{[a,b,c]\ Z\text{-reduced}\\ b^2-4ac=m^2\\ \gcd(a,b,c)=1}}\frac{m}{b}
\;=\; m\sum_{d\mid m}\frac{\mu(d)}{d}\,\psi\!\Big(\frac{m}{2d}\Big)
\;-\; \varphi(m)\sum_{p\mid m}\frac{\log p}{p-1}\, .
\end{equation*}
\end{corollary}

\begin{corollary} In case $m=p$, with $p$ an odd prime, we obtain

$$
\sum_{\substack{b^2-4ac=p^2\\ \gcd(a,b,c)=1\\ a,c,\, a+b+c>0}}\frac{p^3}{3b(b+2a)(b+2c)}
\;+\; \sum_{\substack{[a,b,c]\ Z\text{-reduced}\\ b^2-4ac=p^2\\ \gcd(a,b,c)=1}}\frac{p}{b}
\;=\; p\psi(\frac{p}{2})-\psi(\frac{1}{2})- \log p.
$$
\end{corollary}

\begin{example}
For $m=p=3$ we have only one $Z$-reduced form $[2,5,2]$, so  we obtain $$0.37436765658297927138...+\frac{3}{5}=0.97436765658297927138...\approx$$
$$\approx 3\psi(\frac{3}{2})-\psi(\frac{1}{2})-\log 3=0.9743676592890432...$$
\end{example}

\section{Acknowledgements}

In 2024, I discovered a simple telescopic proof \cite{kalinin2024mordell} of \eqref{hurwitz} without knowing all the above story about topographs that I learned recently, thanks to users of MathOverflow, 

{\tiny \url{https://mathoverflow.net/questions/479076/is-it-known-a-sum-over-lattice-parallelograms-of-area-one-is-equal-to-pi}}

I thank Wadim Zudilin, Mikhail Shkolnikov, and Ernesto Lupercio for discussions.

%
%
%

\bibliography{../bibliography.bib}
\bibliographystyle{abbrv}

\end{document}